\documentclass[12pt]{amsart}

\usepackage{amsmath}
\usepackage{amssymb}
\usepackage{amsopn}
\usepackage{epsfig}
\usepackage{amsfonts}
\usepackage{latexsym}
\usepackage{graphicx}

\newtheorem{theorem}{Theorem}[section]
\newtheorem{lemma}[theorem]{Lemma}
\newtheorem{proposition}[theorem]{Proposition}
\newtheorem{corollary}[theorem]{Corollary}

\theoremstyle{definition}

\theoremstyle{remark}

\numberwithin{equation}{section}

\newcommand{\N}{\ensuremath{\mathcal{N}}}

\newcommand{\A}{\ensuremath{\mathcal{A}}}
\newcommand{\B}{\ensuremath{\mathcal{B}}}

\newcommand{\M}[1]{\mathcal{M}_{#1}}

\newcommand{\set}[1]{\left\{#1\right\}}

\newcommand{\ep}{\varepsilon}
\newcommand{\f}{\infty}

\newcommand{\ra}{\rightarrow}

\begin{document}

\title{On a problem of countable expansions}

\author{Yuru Zou}
\address{College of Mathematics and Computational Science, Shenzhen University, Shenzhen 518060, People's Republic of China}
\email{yrzou@163.com}

\author{Derong Kong}
\address{School of Mathematical Science, Yangzhou University, Yangzhou, JiangSu 225002, People's Republic of China.}
\email{derongkong@126.com}

\date{\today}
\dedicatory{}



\begin{abstract}
For a real number $q\in(1,2)$ and $x\in[0,1/(q-1)]$, the infinite sequence $(d_i)$ is called a \emph{$q$-expansion} of $x$ if
$$
x=\sum_{i=1}^\f\frac{d_i}{q^i},\quad d_i\in\{0,1\}\quad\textrm{for all}~ i\ge 1.
$$
For $m=1, 2, \cdots$ or $\aleph_0$ we denote by $\B_m$ the set of $q\in(1,2)$ such that there exists $x\in[0,1/(q-1)]$ having exactly $m$ different $q$-expansions. It was shown by Sidorov \cite{Sidorov_2009} that  $q_2:=\min \B_2\approx1.71064$, and later asked by Baker \cite{Baker_2015} whether $q_2\in\B_{\aleph_0}$? In this paper we provide a negative  answer to this question and conclude that $\B_{\aleph_0} $ is not a closed set. In particular, we give a complete description of $x\in[0,1/(q_2-1)]$ having exactly two different $q_2$-expansions.
\end{abstract}
\maketitle

\section{Introduction}\label{sec: Introduction}
Given $q\in(1,2)$ and a real $x\in I_q:=[0,1/(q-1)]$ we call the infinite sequence $(d_i)$  a $q$-expansion of $x$ if
$$
x=\sum_{i=1}^\f\frac{d_i}{q^i}, \quad d_i\in\{0,1\}, i\ge 1.
$$
Expansions in non-integer bases were pioneered by R\'{e}nyi \cite{Renyi_1957} and Parry \cite{Parry_1960}. It is well known that for each $q\in(1,2)$ almost every $x\in I_q$ has uncountably many $q$-expansions (see, e.g., \cite{ Dajani_DeVries_2007, Erdos_Joo_Komornik_1990, Sidorov_2003}). In particular, for $q\in(1,q_G)$ all except  two endpoints of $I_q$ have  a  continuum of $q$-expansions, where $q_G=(1+\sqrt{5})/2$. However, for $q>q_G$ there exists infinitely many numbers $x\in I_q$ having a unique $q$-expansion  (see \cite{Erdos_Horvath_Joo_1991}).   Furthermore, Glendinning and Sidorov \cite{Glendinning_Sidorov_2001} showed that there exists a constant $q_{KL}\approx 1.78723$, called the \emph{ Komornik-Loreti constant}, such that  the set $U_q$ of numbers $x$ having a unique $q$-expansion  has positive Hausdorff dimension if $q>q_{KL}$, while $U_q$ is  at most countable if $q<q_{KL}$. Recently, Kong and Li \cite{Kong_Li_2015} gave the Hausdorff dimension of $U_q$ for $q\in(1,2)$ (see also Komornik et al. \cite{Komornik_Kong_Li_2015_1}). For more information we refer to the papers \cite{Komornik_Loreti_2007, DeVries_Komornik_2008, Solomyak_1994} and surveys \cite{Sidorov_2003_1, Komornik_2011}.

Unlike the integer base expansions,  it was discovered by Erd\H{o}s et al. \cite{Erdos_Horvath_Joo_1991, Erdos_Joo_1992} that for $q\in(1,2)$  and  $m=1, 2, \cdots$ or $\aleph_0$ there exists $x\in I_q$ having exactly $m$ different $q$-expansions. We denote by $\B_m$  the set of all such $q$'s, i.e., $\B_m$ is the set of $q\in(1,2)$ such that there exists $x\in I_q$ having exactly $m$ different $q$-expansions.

The following results on $\B_m$ are due to Sidorov and Baker \cite{Baker_2015, Baker_Sidorov_2014, Erdos_Horvath_Joo_1991, Sidorov_2009}.
\begin{theorem}
  \label{t1}
  \begin{enumerate}
    \item The smallest element  of $\B_2$ is $q_2\approx 1.71064$, the appropriate  root of
$$
x^4 =2x^2 +x+1;
$$
    \item The smallest element of $\B_k, k\ge 3,$ is $q_f\approx 1.75488$, the appropriate root of
$$x^3 = 2 x^2 -x+ 1;$$
    \item The smallest element of $\B_{\aleph_0}$ is   $q_G=(1+\sqrt{5})/2$, and the second smallest element of $\B_{\aleph_0}$ is  $q_{\aleph_0}\approx 1.64541$, the appropriate root of
 $$x^6 =x^4 +x^3 +2x^2 +x+1.$$
  \end{enumerate}
\end{theorem}
It was asked by Baker \cite{Baker_2015}  whether $q_2\in\B_{\aleph_0}$?  In this paper we provide a negative answer to this question.
 \begin{theorem}\label{t2}
$q_2\notin\B_{\aleph_0}$.
 \end{theorem}
Sidorov \cite{Sidorov_2009} showed that there exists a sequence $q^{(n)}\in\B_{\aleph_0},n\ge 1,$ strictly decreasing to $q_2$, and later Baker \cite{Baker_2015} proved that $\B_{\aleph_0}\cap(1,q_f]\setminus\{q_2\}$ is a discrete set.  By using  \cite[Theorem 4.5]{Baker_2015} and  Theorem \ref{t2} we have the following structure of $\B_{\aleph_0}$.
 \begin{corollary}\label{c3}
 $\B_{\aleph_0}\cap(1,q_f]$ is a discrete set containing countably infinitely  many elements. Furthermore, $\B_{\aleph_0}$ is not   closed.
 \end{corollary}

  It was shown in \cite[Theorem 4.1]{Baker_2015} that if $x\in I_{q_2}$ has uncountable $q_2$-expansions, then $x$ has  a  continuum of $q_2$-expansions. By using Theorems \ref{t1} and \ref{t2} we have the following corollary.
 \begin{corollary}\label{c4}
  Let $x\in I_{q_2}$. Then  $x$ has a unique $q_2$-expansion,  two $q_2$-expansions, or  a  continuum of $q_2$-expansions.
    \end{corollary}
Denote by $\M{k}$ the set of $x\in I_{q_2}$ having exactly $k$ different $q_2$-expansions. Then   Corollary \ref{c4} says that
$\M{k}=\emptyset$ for $k\ge 3$, and any $x\in I_{q_2}\setminus(\M{1}\cup\M{2})$ has  a  continuum of $q_2$-expansions. The set $\M{1}$ was investigated by Glendinning and Sidorov \cite{Glendinning_Sidorov_2001}. In Theorem \ref{t36} we will give a complete description of $\M{2}$. Interestingly, we find  that $\M{1}$ is the set of all accumulation points of $\M{2}$ (see Corollary \ref{c37}).

 The structure of this paper is arranged as follows. In Section \ref{sec:branching points} we classify the branching points and recall some results on countable expansions. In Section \ref{sec:two expansions} we give a complete description of points in $I_{q_2}$ having exactly two different $q_2$-expansions. The proof of Theorem \ref{t2} will be given in Section \ref{sec:proof of theorem}.

\section{ branching points}\label{sec:branching points}
For $q\in(1,2)$ and   $x\in I_q=[0,1/(q-1)]$ we denote by $\Sigma_q(x)$ the set of all $q$-expansions of $x$, i.e.,
$$
\Sigma_q(x):=\left\{(d_i)\in\{0,1\}^\f: \sum_{i=1}^\f \frac{d_i}{q^i}=x\right\},
$$
where $\{0,1\}^\f$ means the set of  sequences $(c_i)$ with $c_i\in\{0,1\}$ for all $i\ge 1$. We will always write $\Sigma(x)$ instead of $\Sigma_q(x)$ if   no confusion arises for $q$.

For $n\ge 1$ let $\{0,1\}^n$ be the set of words $c_1\cdots c_n$  with $c_i\in\{0,1\}$, and we write for $\{0,1\}^*$ the set of all finite words  $c_1\cdots c_n\in\{0,1\}^n$ for all $n\ge 1$. For two finite words $c_1\cdots c_m, d_1\cdots d_n\in\{0,1\}^*$ we denote by $c_1\cdots c_m d_1\cdots d_n$ their concatenation. In particular, we write for $(c_1\cdots c_m)^k$ and $(c_1\cdots c_m)^\f$ the concatenations of the word $c_1\cdots c_m$ to itself $k$ times and infinitely many times, respectively. Furthermore, we write for $\overline{c_1\cdots c_n}:=(1-c_1)\cdots (1-c_n)$ the \emph{reflection} of the word $c_1\cdots c_n$, and denote by
$(\overline{c_i}):=(1-c_i)$ the reflection of the sequence $(c_i)$.

For $q\in(1,2)$  we consider the following expanding maps
\begin{equation*}
\left\{
\begin{array}{ll}
T_{q,0}(x):=q x &\textrm{if}\quad 0\le x\le1/(q^2-q),\\
 T_{q,1}(x):=qx-1 &\textrm{if}\quad 1/q\le x\le 1/(q-1).
\end{array}
\right.
\end{equation*}
Note that $1/q<1/(q^2-q)$, and the  interval
 $$S_q:=\left[\frac{1}{q}, \frac{1}{q^2-q}\right]$$
   is called the \emph{switch region} of $\{T_{q,0}, T_{q,1}\}$ (see \cite{Dajani_Kraaikamp_2003}). This is because for $x\in S_q$ we have a choice  between $T_{q,0}$ and $T_{q,1}$. For a point $x\in I_q$, if $|\Sigma(x)|>1$, then there exists a word $d_1\cdots d_n\in\{0,1\}^*$ such that
$$
T_{q, d_1\cdots d_n}(x):=
T_{q,d_1}\circ\cdots\circ T_{q,d_n}(x)\in S_q.
$$
 Here $|A|$ denotes the cardinality of a set $A$. In particular, for $n=0$ we  set $T_{q, d_1\cdots d_n}$ as the identity map.

 For $q\in(1,2)$  we classify the points in $S_q$ in the following way:
\begin{itemize}
  \item Let $ \A_1(q)$ be the set of   points $x\in S_q$ satisfying
   $$|\Sigma(T_{q, 0}(x))|<\f\quad\textrm{and}\quad |\Sigma(T_{q, 1}(x))|<\f;$$
  \item Let $ \A_2(q)$ be the set of   points $x\in S_q$ satisfying
  $$|\Sigma(T_{q, s}(x))|<\f\quad\textrm{ and}\quad |\Sigma(T_{q, 1-s}(x))|=\f$$
  for some $s\in\{0,1\}$;
  \item Let $ \A_3(q)$ be the set of   points $x\in S_q$ satisfying
    $$|\Sigma(T_{q, 0}(x))|=\f\quad\textrm{and}\quad|\Sigma(T_{q, 1}(x))|=\f.$$
\end{itemize}
Then $S_q=\bigcup_{i=1}^3 \A_i(q)$.

%
%
%

\begin{def}
  \label{def:1}
  Let $x\in I_{q }$ with $|\Sigma(x)|=\f$. The point $T_{q,d_1\cdots d_n}(x)$ is called a \emph{branching point} of $x$ if
  $T_{q, d_1\cdots d_n}(x)\in \A_2(q)\cup\A_3(q).$
\end{def}

Recall from \cite{Baker_2015} that a point $x\in I_{q}$ with $|\Sigma(x)|=\f$ is called a {\emph{$q$-null infinite point}} if all of its branching points    belong to $\A_2(q)$. Clearly, if $x$ is a $q$-null infinite point, then so are its branching points.

For $q\in(1,2)$ let
$$
J_q:=\left[\frac{q+q^2}{q^4-1},\frac{1+q^3}{q^4-1}\right].
$$
The following lemma is shown by Baker \cite[Lemmas 2.7 and 3.1]{Baker_2015}.
\begin{proposition}\label{p22}
$q\in \B_{\aleph_0}\cap(q_G,q_f]$ if and only if $\A_2(q)\cap J_{q}$ contains a $q$-null infinite point.
\end{proposition}

\section{two $q_2$-expansions}\label{sec:two expansions}

In the remainder part of the paper we will fix $q=q_2\approx 1.71064$. By Theorem \ref{t1} it follows that   points in $ I_{q_2}$ can only  have a unique $q_2$-expansion, two $q_2$-expansions, countably infinitely   many $q_2$-expansions, or  a  continuum of  $q_2$-expansions. In this paper we will show that the third case can not occur, i.e., points in $I_{q_2}$ can not have countably infinitely    many $q_2$-expansions.

Recall in Section \ref{sec: Introduction} that $ \M{k}$ is the set of $x\in I_{q_2}$ having exactly $k$ different $q_2$-expansions. We denote by $ \M{k}'$ the set of corresponding $q_2$-expansions $(d_i)$ satisfying
$$((d_i))_{q_2}:=\sum_{i=1}^\f\frac{d_i}{q_2^i}\in\M{k}.$$
We point out that a number $x\in\M{k}$ corresponds to $k$ different $q_2$-expansions in $\M{k}'$.

Note by Theorem \ref{t1} that for $q=q_2$ if $|\Sigma(x)|<\f$, then $|\Sigma(x)|=1$ or $2$, i.e., $x\in\M{1}\cup\M{2}$.
 The following lemma for $\M{1}$ was shown by Glendinning and Sidorov \cite[Theorem 2]{Glendinning_Sidorov_2001}.
\begin{lemma}
  \label{l31}
      $$\M{1}=\set{(0^\f)_{q_2}, ({1^\f})_{q_2}}\cup\bigcup_{k=0}^\f\set{ (0^k(10)^\f)_{q_2}, (1^k(01)^\f)_{q_2} }.$$
\end{lemma}

Now we turn to the investigation of $\M{2}$. This will be done by a sequel of lemmas. The complete description of $\M{2}$ will be given in Theorem \ref{t36}. Interestingly, we prove in Corollary \ref{c37} that  the set of all accumulation points of $\M{2}$ is exactly $\M{1}$. Therefore, we conclude that $\M{2}$ is a discrete set containing countably infinitely    many elements. Furthermore, $\M{2}$ is not closed.

Recall that $\A_1=\A_1(q_2)$ is the set of      $x\in S_{q_2}$ such that both $|\Sigma(T_{ 0}(x))|$ and $|\Sigma(T_{ 1}(x))|$ are finite. Here and in the sequel we will   write  $T_s$ instead of $T_{q_2, s}$ for $s\in\set{0,1}$. By Theorem \ref{t1} it follows that $\A_1$ is the set of $x\in S_{q_2}$ such that
$$|\Sigma(T_{ 0}(x))|=|\Sigma(T_{ 1}(x))|=1.$$
This implies that $\A_1\subseteq\M{2}$.

The following lemma for $\A_1 $ was shown in \cite[Proposition 3.1]{Baker_Sidorov_2014}(see also, \cite[Proposition 2.4]{Sidorov_2009}). For self-containedness we give an alternative proof.
\begin{lemma}
  \label{l32}
  $$\A_1 =\set{ (01(10)^\f)_{q_2},  (10(01)^\f)_{q_2}}.$$
\end{lemma}
\begin{proof}
Take $x\in\A_1 $. Then $T_0(x), T_1(x)\in\M{1}$. By Lemma \ref{l31} it follows that $x$ must  be of the form
$$
x= (10^j(10)^\f)_{q_2}= (01^k(01)^\f)_{q_2} \quad\textrm{for some}~ j, k\ge 0.
$$
Note that $q_2>q_G$. One can easily check for $j, k=0, 1$ that
 $$ (10^j (10)^\f)_{q_2}>\frac{1}{q_2(q_2-1)}\quad\textrm{and}\quad (01^k (01)^\f)_{q_2}<1/q_2.$$
   Then
$$
x= (10^j(01)^\f)_{q_2}= (01^k(10)^\f)_{q_2} \quad\textrm{for some}~j, k\ge 1.
$$
Equivalently, $q_2$ should be a positive root of the equation
$$
\frac{1}{q}+\frac{1}{q^{j+1}(q^2-1)}=\frac{1}{q^2}+\cdots+\frac{1}{q^{k+1}}+\frac{1}{q^{k}(q^2-1)}
$$
for some $j, k\ge 1$.
Simplifying the above equation it suffices to show that $q_2$ is a positive root of
\begin{equation}\label{e31}
q^{-j}+q^{-k}+q^2-q-2=0\quad\textrm{for some}~k,j\ge 1.
\end{equation}
One can easily check that $q_2$ satisfies the above equation  for $(j,k)=(1,3)$ or $(j,k)=(3,1)$, and in this case
$$
x= (10(01)^\f)_{q_2} = (01^3(10)^\f)_{q_2} \in\A_1,
$$
or
$$
 x=(10^3(01)^\f)_{q_2}=  (01(10)^\f)_{q_2}\in\A_1.$$
We will finish the proof by showing that $(j,k)=(1,3)$ and $(3,1)$ are the only two cases such that (\ref{e31}) holds for $q=q_2$.

Let
$$f(q)=q^{-j}+q^{-k}+q^2-q-2.$$
Then
$
f(\sqrt{2})\le 0<f(2),
$
and  $f'(q)>0$ for $q\in[\sqrt{2}, 2)$. This implies that Equation (\ref{e31}) has a unique solution in $[\sqrt{2}, 2)$, and we denote it by $q_{j,k}$. The proof will be finished by the following observation:
\renewcommand{\labelenumi}{(\roman{enumi})}
\begin{enumerate}
 \item for each $j\ge 1$ the sequence $q_{j,k}$ is strictly increasing as $k\ra\f$;
  \item for each $k\ge 1$ the sequence $q_{j,k}$ is strictly increasing as $j\ra\f$.
  \end{enumerate}
 By symmetry  we only give the proof of (i). For simplicity we write  $q_k=q_{j,k}$. Then by (\ref{e31}) we have
  \begin{equation*}
  \begin{split}
q_k^{-k}+q_k^{-j}+q_k^2-q_k-2=0&=q_{k+1}^{-k-1}+q_{k+1}^{-j}+q_{k+1}^2-q_{k+1}-2\\
  &<q_{k+1}^{-k}+q_{k+1}^{-j}+q_{k+1}^2-q_{k+1}-2
  \end{split}
  \end{equation*}
  i.e.,
  $$
  f(q_k)<f(q_{k+1}).
  $$
This implies $q_k<q_{k+1}$, since $f$ is strictly increasing in $[\sqrt{2}, 2)$.
\qed
\end{proof}

Based on Lemma \ref{l32} we give a characterization of   $\M{2}$ (see also, \cite{Baker_Sidorov_2014, Sidorov_2009}).
\begin{lemma}\label{l33}
 $x\in{\M{2}} $ if, and only if, there exists a finite word $d_1\cdots d_n\in\{0,1\}^n$ with $n\ge 0$ such that
  $$
  T_{ d_1\cdots d_n}(x)\in\A_1 \quad\textrm{and}\quad T_{ d_1\cdots d_i}(x)\notin S_{q_2}
  $$
  for all $0\le i<n$.
\end{lemma}
\begin{proof}
The sufficiency    follows by Lemma \ref{l32}. For the necessity, we take $x\in\M{2}$, and let $(a_i)$ and $(b_i)$ be the   two $q_2$-expansions of $x$, i.e.,
$$
 ((a_i)) _{q_2} =  ((b_i)) _{q_2} =x.
$$

Let $k\ge 1$ be the least integer such that $a_k\ne b_k$. Then
$$a_{k+1}a_{k+2}\cdots\in{\M{1}'},\quad b_{k+1}b_{k+2}\cdots\in {\M{1}'},$$
and therefore
$$
T_{a_1\cdots a_{k-1}}(x)= (a_{k}a_{k+1}\cdots)_{q_2} = (b_{k}b_{k+1}\cdots)_{q_2} \in\A_1.
$$
Moreover, for any $i<k-1$ we have $T_{a_1\cdots a_i}(x)\notin S_{q_2}$, since otherwise   the point
$
x
$
 will have more than two $q_2$-expansions which contradicts to $x\in\M{2}$.

Therefore, the necessity follows by taking $d_1\cdots d_n=a_1\cdots a_{k-1}$.
\qed
\end{proof}

Note that $((\overline{d_i}))_{q_2}=1/(q_2-1)-((d_i))_{q_2}$. Similar to $\M{1}$ we prove that $\M{2}$ is also symmetric.
\begin{lemma}\label{l34}
$x\in\M{2}$ if and only if $1/(q_2-1)-x\in\M{2}$.
\end{lemma}
\begin{proof}
  Let $x\in\M{2}$. By Lemma \ref{l33} it follows that there exists $d_1\cdots d_n\in\set{0,1}^n$ such that
  $
  T_{d_1\cdots d_n}(x)\in\A_1.
  $
 Then by Lemma \ref{l32} we obtain that
\begin{equation*}
  \begin{split}
    T_{\overline{d_1\cdots d_n}}\left(\frac{1}{q_2-1}-x\right)&=q_2^n\left(\frac{1}{q_2-1}-x\right)-\sum_{i=1}^n q_2^{n-i}(1-d_{n-i+1})\\
    &= \frac{1}{q_2-1}-\left(q_2^n x-\sum_{i=1}^n q_2^{n-i}d_{n-i+1}\right)\\
    &=\frac{1}{q_2-1}-T_{d_1\cdots d_n}(x)\\
    &\in\A_1.
  \end{split}
\end{equation*}
Furthermore, for all $0\le i<n$ we have $T_{d_1\cdots d_i}(x)\notin S_{q_2}$ if, and only if,
$$T_{\overline{d_1\cdots d_i}}\left(\frac{1}{q_2-1}-x\right)=\frac{1}{q_2-1}-T_{d_1\cdots d_i}(x)\notin S_{q_2}.$$
By using Lemma \ref{l33} this implies  $1/(q_2-1)-x\in \M{2}$.
\qed\end{proof}

In terms of Lemma 3.3 we still need to investigate all of those finite words $d_1\cdots d_n$ such that
$$(d_1\cdots d_n c_1 c_2\cdots)_{q_2}\in\M{2}\quad\textrm{with}\quad ((c_i))_{q_2}\in\A_1.$$
By Lemmas \ref{l32} and \ref{l34} it suffices to consider the case for $((c_i))_{q_2}=(01(10)^\f)_{q_2}$.
\begin{lemma}\label{l35}
For $n\ge 2$    let
  $\eta _n:=d_1\cdots d_n(10)^\f \in\M{2}' $ {with} $d_{n-1}d_n=01.$ Then the following statements hold.
   \renewcommand{\labelenumi}{(\Alph{enumi})}
  \begin{enumerate}
   \item If $d_1 d_2=00$, then $0\eta_n \in\M{2}', 1\eta _n\notin\M{2}'$;
    \item  If $d_1 d_2=11$, then $1\eta_n \in\M{2}', 0\eta_n \notin\M{2}'$;
    \item If $d_1d_2=01$, then $0\eta _n, 1\eta _n\in\M{2}'$;
    \item If $d_1d_2= 10$, then $0\eta _n, 1\eta _n\in\M{2}'$.
  \end{enumerate}
\end{lemma}
\begin{proof}
  By symmetry we only prove (A) and (C).

  First we prove (A). Suppose $d_1d_2=00$. Then  $n\ge 3$, and by Lemma \ref{l33} it suffices to prove that
  $$
  f_0((\eta_n)_{q_2})\notin S_{q_2},\quad f_1((\eta_n)_{q_2})\in S_{q_2},
  $$
  where
\begin{equation*}
  f_s(x):=\frac{x+s}{q_2},\quad s\in\{0,1\}.
  \end{equation*}

It is obvious that
$$f_0((\eta_n)_{q_2})< (0001^\f)_{q_2}<\frac{1}{q_2}.$$ This implies $f_0( (\eta_n)_{q_2})\notin S_{q_2}$.
Note that
$$
\frac{1}{q_2}<(01110^\f)_{q_2}\le (0111a_1a_2\cdots)_{q_2}\le (01^\f)_{q_2}=\frac{1}{q_2(q_2-1)},
$$
 for any $(a_i)\in \{0,1\}^\f$. By Lemma \ref{l33} it follows that the
word $111$ can not appear in $\eta_n =00d_3\cdots d_{n}(10)^\f$, and therefore,
$$
(\eta_n)_{q_2}=  (00d_3\cdots d_{n}(10)^\f)_{q_2}\le  (00(110)^\f)_{q_2}.
$$
This implies
$$
\frac{1}{q_2}\le f_1( (\eta_n)_{q_2})= (100d_3\cdots d_n   (01)^\f)_{q_2} \le (100(110)^\f)_{q_2} <\frac{1}{q_2(q_2-1)}.
$$
So, (A) is verified.

Now  we turn to the proof of $(C)$. Suppose $d_1d_2=01$. One can easily check that
$$f_0((\eta_n)_{q_2})< (001^\f)_{q_2}<1/q_2,$$
implying  $f_0( (\eta_n)_{q_2})\notin S_{q_2}$. Note that
$$ \frac{1}{q_2}=(10^\f)_{q_2}\le (1000b_1b_2\cdots)_{q_2}\le(10001^\f)_{q_2}<\frac{1}{q_2(q_2-1)}$$
 for any $(b_i)\in\set{0,1}^\f$. Then  by Lemma \ref{l33} it follows that the word $000$ can not appear in
$\eta_n =01d_3\cdots d_n (10)^\f$, and therefore
$$
 (\eta_n)_{q_2}= (01d_3\cdots d_n(10)^\f)_{q_2}\ge (01(001)^\f)_{q_2}.
$$
This implies that
$$
f_1( (\eta_n)_{q_2})= (101d_3\cdots d_n 01(10)^\f)_{q_2} > (101(001)^\f)_{q_2} >\frac{1}{q_2(q_2-1)}.
$$
 Therefore, (C) holds.
 \qed
\end{proof}

Now we give a complete description of $\M{2}$ based on Lemmas \ref{l32}--\ref{l35}.
\begin{theorem}\label{t36}
 \begin{equation*}
 \begin{split}
 \M{2}=\bigcup_{m=0}^\f\bigcup_{k=1}^\f\set{(0^m \ep_k)_{q_2}, (1^m \ep_k)_{q_2}, (\overline{0^m\ep_k})_{q_2},(\overline{1^m\ep_k})_{q_2} },
 \end{split}
 \end{equation*}
 where
\begin{equation}\label{e32}
\ep_k:=(01)^k(10)^\f\quad\textrm{for}~k\ge 1.
\end{equation}
\end{theorem}
\begin{proof}
By Lemma \ref{l35} it follows that
$$
0^m\ep_k, 1^m\ep_k\in\M{2}'
$$
for all $m\ge 0$ and $k\ge 1$. Then by Lemmas \ref{l34} we obtain the ``$\supseteq$"  part.

For the ``$\subseteq$" part, we take $x\in\M{2}$. Then by Lemmas \ref{l32} and \ref{l33}  there exists a word $d_1\cdots d_n$ such that
$$
T_{d_1\cdots d_n}( x)\in\A_1=\set{ (01(10)^\f)_{q_2},   (10(01)^\f)_{q_2}}
$$
and
$$
T_{d_1\cdots d_i}( x)\notin S_{q_2}
$$
for any $0\le i<n$.

Without loss of generality we assume $T_{d_1\cdots d_n}( x)= (01(10)^\f)_{q_2}$. Then
$$
x=(d_1\cdots d_n01(10)^\f)_{q_2},
$$
and hence by Lemma \ref{l35} it follows that
$$
x=(d_1\cdots d_n01(10)^\f)_{q_2}\in\bigcup_{m=0}^\f\bigcup_{k=1}^\f\set{(0^m\ep_k)_{q_2}, (1^m\ep_k)_{q_2}}.
$$
\qed
\end{proof}
By Lemma \ref{l31} and Theorem \ref{t36} we have the following connection between $\M{1}$ and $\M{2}$.
\begin{corollary}\label{c37}
The set of all accumulation points of $\M{2}$ is $\M{1}$.
\end{corollary}
By Theorem \ref{t36} and Corollary \ref{c37} it follows that $\M{2}$ is a discrete set containing countably infinitely    many elements. Furthermore, $\M{2}$ is not closed. This is opposite to $\M{1}$, since we know by Lemma \ref{l31} that $\M{1}$ is not discrete but closed.

\section{Proof of Theorem \ref{t2}}\label{sec:proof of theorem}
In this section we will prove $q_2\notin\B_{\aleph_0}$.
In terms of Proposition \ref{p22},   it suffices to prove that  $\A_2\cap J_{q_2}$ contains no $q_2$-null infinite points, where
\begin{equation}\label{eq:41}
\begin{split}
J_{q_2}=[ ((0110)^\f)_{q_2}, ((1001)^\f)_{q_2}]\approx[0.613089, 0.794085].
\end{split}
\end{equation}
Recall that $\A_2$ is the set of $x\in S_{q_2}$ such that
$|\Sigma(T_s(x))|<\f$ and $|\Sigma(T_{1-s}(x))|=\f$ for some $s\in\set{0,1}$. By Theorem \ref{t1} it follows that $T_s(x)\in\M{1}\cup\M{2}$, and therefore
$$
\A_2\subseteq\bigcup_{s=0}^1 T_s^{-1}(\M{1}\cup\M{2}).
$$

 In the following lemma we give a characterization of $\A_2 \cap J_{q_2}$. For a real number $r$ and a set $A$ we define
 $r-A=\set{r-a: a\in A}$.
\begin{lemma}\label{l41}
$$\A_2\cap J_{q_2}=\bigcup_{m=1}^4 \left(E_m\cup\Big(\frac{1}{q_2-1}-E_m\Big)\right),$$ where
\begin{equation*}
E_m :=  \bigcup_{k=1}^\f\set{(01^{m+1}\ep_k)_{q_2 }, (10^m\ep_k)_{q_2 }}\setminus\set{  (10\ep_1)_{q_2}}
\end{equation*}
for $m=1, 3$, and
\begin{equation*}
E_m :=\set{(01^m(10)^\f)_{q_2 }} \cup  \bigcup_{k=1}^\f\set{(01^{m+1}\ep_k)_{q_2 }, (10^m\ep_k)_{q_2 }}.
\end{equation*}
 for $m=2,4$.
\end{lemma}
\begin{proof}
Note that
$J_{q_2}=[ ((0110)^\f)_{q_2}, ((1001)^\f)_{q_2}]$.
This yields
$$
(01^5(10)^\f)_{q_2}>((1001)^\f)_{q_2}, \quad(10^5(01)^\f)_{q_2}<((0110)^\f)_{q_2}.
$$
 Then by Lemma \ref{l31}  we obtain that
 \begin{equation}\label{eq:1}
  \bigcup_{s=0}^1 T_s^{-1}(\M{1} )\cap J_{q_2}= \bigcup_{m=1}^4\set{(01^m(10)^\f)_{q_2}, (10^m(01)^\f)_{q_2}}.
  \end{equation}
Furthermore, by Theorem \ref{t36} and (\ref{eq:41}) it follows that
\begin{equation}\label{eq:2}
\begin{split}
\bigcup_{s=0}^1 T_s^{-1}(\M{2} )\cap J_{q_2}&= \bigcup_{m=1}^4\bigcup_{k=1}^\f\set{(01^{m+1}\ep_k)_{q_2}, (10^m\ep_k)_{q_2}}\setminus\set{(10\ep_1)_{q_2}}\\
&\quad\cup\bigcup_{m=1}^4\bigcup_{k=1}^\f\set{ (\overline{01^{m+1}\ep_k})_{q_2}, (\overline{10^m\ep_k})_{q_2}}\setminus\set{(\overline{10\ep_1})_{q_2}}.
\end{split}
\end{equation}

Note by Lemma \ref{l32}   that
\begin{equation}\label{eq:3}
\A_1=\set{(01(10)^\f)_{q_2}, (01^3(10)^\f)_{q_2}},
\end{equation} and by Theorem  \ref{t1} that
$$\A_1\cup\A_2 =\bigcup_{s=0}^1 T_s^{-1}(\M{1}\cup\M{2}).$$
Therefore, the lemma follows by using (\ref{eq:1})--(\ref{eq:3}) in the following equation:
$$
\A_2\cap J_{q_2}=\left(\Big(\bigcup_{s=0}^1 T_s^{-1}(\M{1})\cap J_{q_2}\Big)\cup\Big(\bigcup_{s=0}^1 T_s^{-1}\M{2}\cap J_{q_2}\Big)\right)\setminus\A_1.
$$
\qed
\end{proof}

By using Lemma \ref{l41} one can easily verify  the following monotonicity of the elements
 in $\A_2\cap J_{q_2}$.
\begin{lemma}\label{l42}
\begin{enumerate}
  \item For each $m\ge 1$, we have for $k\ra\f$ that
 \begin{equation*}
 \begin{split}
  &(\overline{10^m\ep_k})_{q_2}\nearrow(01^m(10)^\f)_{q_2}, \quad(01^{m+1}\ep_k)_{q_2}\searrow (01^m(10)^\f)_{q_2},\\
 &(\overline{01^{m+1}\ep_k})_{q_2}\nearrow(\overline{01^m(10)^\f})_{q_2},\quad (10^m\ep_k)_{q_2}\searrow(\overline{01^m(10)^\f})_{q_2};
\end{split}
\end{equation*}

  \item $  \A_2 \cap J_{q_2}\subseteq{\mathcal{H}}$, where
  $$
{\mathcal{H}}:=\bigcup_{m=1}^4 \big[ (\overline{10^m\ep_1})_{q_2}, (01^{m+1}\ep_1)_{q_2}\big]\cup\bigcup_{m=1}^4 \big[ (\overline{01^{m+1}\ep_1})_{q_2},  (10^m\ep_1)_{q_2}\big].
  $$\end{enumerate}
\end{lemma}

Let $\N$ be the set of all $q_2$-null infinite points. The following lemma says that $\N$ is symmetric.
\begin{lemma}\label{l43}
  $x\in\N$ if and only if $1/(q_2-1)-x\in\N$.
\end{lemma}
\begin{proof}
Note that for $s\in\{0,1\}$ we have
$$
T_{1-s}\left(\frac{1}{q_2-1}-x\right)=\frac{1}{q_2-1}-T_s(x).
$$
This means that $(d_i)$ is a $q_2$-expansion of $T_s(x)$ if and only if $(\overline{d_i})$ is a $q_2$-expansion of $T_{1-s}(1/(q_2-1)-x)$. Therefore,
$$
\left|\Sigma\Big(T_{1-s}\big(\frac{1}{q_2-1}-x\big)\Big)\right|=|\Sigma(T_s(x))|.
$$
This implies that
$$x\in\A_2 \quad\Longleftrightarrow\quad\frac{1}{q_2-1}-x\in\A_2.$$

Furthermore, one can show that
$$
T_{\overline{d_1\cdots d_n}}\Big(\frac{1}{q_2-1}-x\Big)=\frac{1}{q_2-1}-T_{d_1\cdots d_n}(x),
$$
and therefore
$$
T_{d_1\cdots d_n}(x)\in\A_2  \quad\Longleftrightarrow\quad
T_{\overline{d_1\cdots d_n}}\Big(\frac{1}{q_2-1}-x\Big) \in\A_2.
$$
Hence, the lemma follows by the definition of $q_2$-null infinite points.
\qed\end{proof}

In order to prove Theorem \ref{t2} we need some   numerical calculation. By (\ref{e32}) we obtain
\begin{equation}\label{e41}
 \begin{split}
 (01^{m+1}\ep_k)_{q_2}&=\frac{q_2^{m+2k+2}+q_2^{m+2k+1}-q_2^{2k+1}+q_2-1}{q_2^{m+2k+2}(q_2^2-1)},\\
 (10^m\ep_k)_{q_2}&=\frac{q_2^{m+2k+2}-q_2^{m+2k}+q_2^{2k}+q_2-1}{q_2^{m+2k+1}(q_2^2-1)}.
 \end{split}
\end{equation}

Then by Lemma \ref{l42} we give the approximate values for intervals of ${\mathcal{H}}$ in Table \ref{tab:1}.
\begin{table}[h!]
\caption{Approximate values for intervals of ${\mathcal{H}}$}\label{tab:1}
\begin{center}
\begin{tabular}{|c|c|c|}
\hline
  m & $[(\overline{10^m\ep_1})_{q_2}, (01^{m+1}\ep_1)_{q_2}]$ & $[(\overline{01^{m+1}\ep_1})_{q_2}, (10^m\ep_1)_{q_2}]$ \\\hline
  1& [{0.602117, 0.670382}]& [{0.736792, 0.805057}] \\
  2 & [{0.693711, 0.733617}] &[{0.673557, 0.713464}] \\
  3& [{0.747254, 0.770582}]& [{0.636592, 0.65992}]\\
  4 & [{0.778554, 0.792191}]& [{0.614983, 0.62862}] \\
  \hline
\end{tabular}
\end{center}
\end{table}

Now we turn to the proof of $q_2\notin\B_{\aleph_0}$. By Proposition \ref{p22} it suffices to prove that $\A_2\cap J_{q_2}$ contains no $q_2$-null infinite points.  Then
by Lemmas \ref{l41} and \ref{l43} we only needs to show that $E_{m} \cap \N =\emptyset$ for $m=1,2,3, 4$, where
$E_m$ is defined in Lemma \ref{l41}.

Our approach  to prove $E_m\cap\N=\emptyset$ is as follows.  If $x\in E_m\cap \N$, then $T_{d_1\cdots d_n}(x)\in J_{q_2}$ implies that $T_{d_1\cdots d_n}(x)\in \A_2$. So, to exclude a point  $x\in E_m$ from $\N$ it suffices to prove that there exists a word $d_1\cdots d_n$ such that $T_{d_1\cdots d_n}(x)\in J_{q_2}\setminus \A_2$.

\begin{lemma}\label{l44}
$E_m\cap \N =\emptyset$ for $m=1$ and $3$.
\end{lemma}
 \begin{proof}
Recall from Lemma \ref{l41} that
\begin{equation*}
E_m :=  \bigcup_{k=1}^\f\set{(01^{m+1}\ep_k)_{q_2 }, (10^m\ep_k)_{q_2 }}\setminus\set{  (10\ep_1)_{q_2}}
\end{equation*}
for $m=1,3$.
By (\ref{e41}) and using $q_2^4=2q_2^2+q_2+1$ it follows that for any $k\ge 1$ we have
\begin{equation*}
\begin{split}
 &T_{(10)^{k-1}0^{3}1}((01^{2}\ep_k)_{q_2})\\
 &=q_2^{2k+2}\left(\frac{q_2^{2k+3}+q_2^{2k+2}-q_2^{2k+1}+q_2-1}{q_2^{2k+3}(q_2^2-1)}-\Big(\frac{1}{q_2}+\frac{1-q^{-2k+2}}{q^4(q^2-1)}\Big)\right)\\
 &=\frac{-q_2^{2k-1}(q_2-1)(q_2^4-2q_2^2-q_2-1)+2q_2-1}{q_2(q_2^2-1)}\\
 &=\frac{2q_2-1}{q_2^3-q_2} (\approx 0.734788)~\in J_{q_2}\setminus{\mathcal{H}},
 \end{split}
\end{equation*}
 where the last inclusion follows by Table \ref{tab:1}. So, by Lemma \ref{l42} and Proposition \ref{p22} it follows that $ (01^{2}\ep_k)_{q_2} \notin \N $ for any $k\ge 1$.

 Similarly, by (\ref{e41}) and using $q_2^4=2q_2^2+q_2+1$ one can show that
 $$
T_{(10)^{k+1}01}((01^4\ep_k)_{q_2})=T_{(10)^k1^20}((10^3\ep_k)_{q_2})=\frac{2q_2-1}{q_2^3-q_2}\in J_{q_2}\setminus{\mathcal{H}}
$$
for any $k\ge 1$.
This implies $(01^4\ep_k)_{q_2}, (10^3\ep_k)_{q_2}\notin\N$. Furthermore,
 $$
 T_{(10)^{k-2}1^40}((10\ep_k)_{q_2})=\frac{2q_2-1}{q_2^3-q_2}\in J_{q_2}\setminus{\mathcal{H}},
 $$
 for any $k\ge 2$, implying  $(10\ep_k)_{q_2}\notin\N$.
\qed\end{proof}

\begin{lemma}\label{l45}
  $E_m\cap\N_{q_2}=\emptyset$ for $m=2$ and $4$.
\end{lemma}
\begin{proof}
Recall from Lemma \ref{l41} that
\begin{equation*}
E_m :=\set{(01^m(10)^\f)_{q_2 }} \cup  \bigcup_{k=1}^\f\set{(01^{m+1}\ep_k)_{q_2 }, (10^m\ep_k)_{q_2 }}.
\end{equation*}
for $m=2,4$.

 By (\ref{e41}) and using  $q_2^4=2q_2^2+q_2+1$  it follows that
\begin{equation}\label{e42}
T_{0^21}((01^3\ep_k)_{q_2})=\frac{q_2^3-q_2-2+q_2^{-2k}-q_2^{-2k-1}}{q_2^2-1}\ra\frac{q_2^3-q_2-2}{q_2^2-1}
\end{equation}
as $k\ra\f$. Then, by Table \ref{tab:1} it follows that
$$T_{0^21}((01^2(10)^\f)_{q_2})=\frac{q_2^3-q_2-2}{q_2^2-1}(\approx 0.672386)~\in J_{q_2}\setminus{\mathcal{H}}.$$
This implies $(01^2(10)^\f)_{q_2}\notin\N$.

Note by (\ref{e42}) that $T_{0^21}((01^3\ep_k)_{q_2})$ decreases   as $k\ra\f$. Then by Table \ref{tab:1} and numerical calculation one can show that $$T_{0^21}((01^3\ep_k)_{q_2})\in J_{q_2}\setminus \mathcal{H}$$ for all $k\ge 5$.
So, by Proposition \ref{p22} and Lemma \ref{l42} it follows that $(01^3\ep_k)_{q_2}\notin\N$ for all $k\ge 5$.

In the following we will prove $(01^3\ep_k)_{q_2}\notin\N$ for $k\le 4$.  First we consider the case $k=4$.
By (\ref{e42}) and Table \ref{tab:1} it follows that
$$T_{0^21}((01^3\ep_4)_{q_2})\approx 0.675327\in [(\overline{10^3\ep_1})_{q_2}, (10^2\ep_1)_{q_2}].$$
Then by using the monotonicity in  Lemma \ref{l42} one can show that
$$
(\overline{01^3\ep_1})_{q_2}<T_{0^21}((01^3\ep_4)_{q_2})<(\overline{01^3\ep_2})_{q_2},
$$
which, together with Lemma \ref{l41}, implies that $T_{0^21}((01^3\ep_4)_{q_2})\in J_{q_2}\setminus\A_2$. Therefore, $(01^3\ep_4)_{q_2}\notin\N$.

Similarly, one can show by using Lemmas  \ref{l41} and \ref{l42} that  all of these numbers
$T_{0^21}((01^3\ep_1)_{q_2})\approx 0.746083$, $T_{0^21}((01^3\ep_2)_{q_2})\approx 0.69757$ and $T_{0^21}((01^3\ep_3)_{q_2})\approx 0.680992$
belong to $J_{q_2}\setminus\A_2$. Hence,  $(01^3\ep_k)_{q_2}\notin\N$ for all $k\ge 1$.

Symmetrically, by (\ref{e41}) and using $q_2^4=2q_2^2+q_2+1$ we obtain
$$
T_{10^21}((01^5\ep_k)_{q_2})=\frac{2q_2-1+q_2^{-2k}-q_2^{-2k-1}}{q^3_2-q_2}\ra\frac{2q_2-1}{q_2^3-q_2}\in J_{q_2}\setminus{\mathcal{H}}
$$
as $k\ra\f.$ This yields that
$(01^4(10)^\f)_{q_2} \notin\N$. In a similar way as above we can prove that $(01^5\ep_k)_{q_2}\notin\N$ for all $k\ge 1$.

Furthermore, the proof of
$$(10^2\ep_k)_{q_2}, (10^4\ep_k)_{q_2}\notin\N\quad\textrm{for all}~k\ge 1,$$ can be done in a similar way by observing that
$$
T_{1^2 0}((10^2\ep_k)_{q_2})=\frac{2q_2-1+q_2^{2-2k}-q_2^{1-2k}}{q_2^3-q_2}\ra\frac{2q_2-1}{q_2^3-q_2}\in J_{q_2}\setminus{\mathcal{H}},
$$
and
$$
T_{01^20}((10^4\ep_k)_{q_2})=\frac{q_2^3-q_2-2+q_2^{1-2k}-q_2^{-2k}}{q_2^2-1}\ra\frac{q_2^3-q_2-2}{q_2^2-1}\in J_{q_2}\setminus{\mathcal{H}}
$$
as $k\ra\f$.
\qed\end{proof}

\begin{proof}
 By Lemmas \ref{l44}--\ref{l45} it follows that
$$
\N\cap\bigcup_{m=1}^4 E_m=\emptyset.
$$
Then by Lemmas \ref{l41} and \ref{l43} we have
$$
 \N\cap\A_2 \cap J_{q_2} =\emptyset.
$$
Therefore, we conclude by Proposition \ref{p22} that $q_2\notin\B_{\aleph_0}$.
\qed\end{proof}

\section*{Acknowledgments}
The authors thank the anonymous referee for many useful remarks. In particular, the authors thank Simon Baker for some suggestions and   references.
The first author was supported by the NSFC  no 11201312, no 61373087, no 61272252; the Foundation for
Distinguished Young Teachers in Guangdong, China no Yq2013144.
The second author was supported by the NSFC no  11401516, no 11271137 and Jiangsu Province Natural
Science Foundation for the Youth no BK20130433.


\begin{thebibliography}{10}
\expandafter\ifx\csname url\endcsname\relax
  \def\url#1{\texttt{#1}}\fi
\expandafter\ifx\csname urlprefix\endcsname\relax\def\urlprefix{URL }\fi
\expandafter\ifx\csname href\endcsname\relax
  \def\href#1#2{#2} \def\path#1{#1}\fi




 \bibitem{Baker_2015}
S.~Baker, \href{http://dx.doi.org/10.1016/j.jnt.2014.08.003}{On small bases
  which admit countably many expansions}, J. Number Theory 147 (2015) 515--532.


\bibitem{Baker_Sidorov_2014}
S.~Baker, N.~Sidorov, Expansions in non-integer bases: lower order revisited,
  Integers 14 (2014) Paper No. A57, 15.


\bibitem{Dajani_DeVries_2007}
K.~Dajani, M.~de~Vries, \href{http://dx.doi.org/10.4171/JEMS/76}{Invariant
  densities for random {$\beta$}-expansions}, J. Eur. Math. Soc.   9~(1)
  (2007) 157--176.


\bibitem{Dajani_Kraaikamp_2003}
K.~Dajani, C.~Kraaikamp,
  \href{http://dx.doi.org/10.1017/S0143385702001141}{Random
  {$\beta$}-expansions}, Ergodic Theory Dynam. Systems 23~(2) (2003) 461--479.


\bibitem{DeVries_Komornik_2008}
M.~de~Vries, V.~Komornik,
  \href{http://dx.doi.org/10.1016/j.aim.2008.12.008}{Unique expansions of real
  numbers}, Adv. Math. 221~(2) (2009) 390--427.


\bibitem{Erdos_Joo_Komornik_1990}
P.~Erd\H{o}s, I.~Jo\'{o}, V.~Komornik, Characterization of the unique
  expansions $1=\sum_{i=1}^\infty q^{-n_i}$ and related problems, Bull. Soc.
  Math. France 118 (1990) 377--390.


 \bibitem{Erdos_Horvath_Joo_1991}
P.~Erd{\H{o}}s, M.~Horv{\'a}th, I.~Jo{\'o},
  \href{http://dx.doi.org/10.1007/BF01903963}{On the uniqueness of the
  expansions {$1=\sum q^{-n_i}$}}, Acta Math. Hungar. 58~(3-4) (1991) 333--342.



 \bibitem{Erdos_Joo_1992}
P.~Erd{\H{o}}s, I.~Jo{\'o}, On the number of expansions {$1=\sum q^{-n_i}$},
  Ann. Univ. Sci. Budapest. E\"otv\"os Sect. Math. 35 (1992) 129--132.



\bibitem{Glendinning_Sidorov_2001}
P.~Glendinning, N.~Sidorov, Unique representations of real numbers in
  non-integer bases, Math. Res. Lett. 8 (2001) 535--543.


\bibitem{Komornik_2011}
V.~Komornik, Expansions in noninteger bases, Integers 11B (2011) Paper No. A9,
  30.


\bibitem{Komornik_Kong_Li_2015_1}
V.~Komornik, D.~Kong, W.~Li, Hausdorff dimension of univoque sets and devil's
  staircase, arXiv:1503.00475.


\bibitem{Komornik_Loreti_2007}
V.~Komornik, P.~Loreti, \href{http://dx.doi.org/10.1016/j.jnt.2006.04.006}{On
  the topological structure of univoque sets}, J. Number Theory 122  (2007)
  157--183.


\bibitem{Kong_Li_2015}
D.~Kong, W.~Li, Hausdorff dimension of unique beta expansions, Nonlinearity
  28~(1) (2015) 187--209.


\bibitem{Parry_1960}
W.~Parry, On the $\beta$-expansions of real numbers, Acta Math. Acad. Sci.
  Hungar. 11 (1960) 401--416.

\bibitem{Renyi_1957}
A.~R\'{e}nyi, Representations for real numbers and their ergodic properties,
  Acta Math. Acad. Sci. Hungar. 8 (1957) 477--493.






\bibitem{Sidorov_2003}
N.~Sidorov, \href{http://dx.doi.org/10.2307/3647804}{Almost every number has a
  continuum of {$\beta$}-expansions}, Amer. Math. Monthly 110~(9) (2003)
  838--842.


\bibitem{Sidorov_2003_1}
N.~Sidorov, \href{http://dx.doi.org/10.1017/CBO9780511546716.010}{Arithmetic
  dynamics}, in: Topics in dynamics and ergodic theory, Vol. 310 of London
  Math. Soc. Lecture Note Ser., Cambridge Univ. Press, Cambridge, 2003, pp.
  145--189.





\bibitem{Sidorov_2009}
N.~Sidorov, \href{http://dx.doi.org/10.1016/j.jnt.2008.11.003}{Expansions in
  non-integer bases: lower, middle and top orders}, J. Number Theory 129
  (2009) 741--754.




\bibitem{Solomyak_1994}
B.~Solomyak, \href{http://dx.doi.org/10.1112/plms/s3-68.3.477}{Conjugates of
  beta-numbers and the zero-free domain for a class of analytic functions},
  Proc. London Math. Soc. (3) 68~(3) (1994) 477--498.


\end{thebibliography}
\end{document}